\documentclass[11pt]{amsart}

\usepackage{graphicx,amssymb} 
\usepackage{epsfig}
\usepackage{a4wide}
\usepackage{color}
\usepackage{euscript}

\begin{document}
 
\title[Maximally singular solutions of Laplace equations]
{Maximally singular weak solutions\\ of Laplace equations}

\author[J. P. Mili\v{s}i\'{c} and D.\ \v Zubrini\'c]{Josipa-Pina Mili\v{s}i\'{c} and Darko \v Zubrini\'c}
\address{University of Zagreb, Faculty of Electrical Engineering and Computing,
Department of Applied Mathematics, Unska 3,
10000 Zagreb, Croatia}
\email{darko.zubrinic@fer.hr, pina.milisic@fer.hr}

\date{\today}

\thanks{J.\ P.\ Mili\v si\'c acknowledges support from   
the Croatian Science Foundation IP-2013-11-3955,
and the Austrian-Croatian Project of the Austrian Exchange Service (\"OAD) and the Ministry of Science and Education of the Republic of Croatia (MZO). 
D.\ \v Zubrini\'c acknowledges support from the Croatian Science Foundation IP-2014-09-2285 and by the Franco-Croatian
PHC-COGITO project.} 



\begin{abstract}
It is known that there exists an explicit function $F$ in $L^2(\Omega)$, where $\Omega$ is a 
given bounded open subset of $\mathbb{R}^N$, such that the corresponding weak solution of the Laplace BVP 
$-\Delta u=F(x)$, $u\in H_0^1(\Omega)$, is {\em maximally singular}; that is,
 the singular set of $u$ (defined in the Introduction) has the Hausdorff dimension equal to $(N-4)^+$. 
 This constant is optimal, i.e., the largest possible. Here, we show that much more is true: 
 when $N\ge5$, there exists $F\in L^2(\Omega)$ such that the corresponding weak solution has 
 the {\em pointwise
concentration of singular set} of $u$, in the sense of the Hausdorff dimension, equal to $N-4$ 
at {\em all} points of $\Omega$.
We also consider the problem of generating weak solutions with the property of {\em contrast}; 
that is, we construct solutions $u$ that are regular (more specifically, of class $C_{loc}^{2,\alpha}$ for arbitrary $\alpha\in(0,1)$) in any
prescribed open subset $\Omega_r$ of $\Omega$,
while they are {\em maximally singular} in its complement $\Omega\setminus\Omega_r$. We indicate several open problems.
\end{abstract}

\maketitle

\newtheorem{theorem}{Theorem}
\newtheorem{cor}{Corollary}
\newtheorem{prop}{Proposition}
\newtheorem{defn}{Definition}
\newtheorem{lemma}{Lemma}

\theoremstyle{remark}
\newtheorem{remark}{Remark}
\newtheorem{example}{Example}

\newcommand{\di}{\displaystyle}  

\font\csc=cmcsc10

\def\esssup{\mathop{\rm ess\,sup}}
\def\essinf{\mathop{\rm ess\,inf}}
\def\wo#1#2#3{W^{#1,#2}_0(#3)}
\def\w#1#2#3{W^{#1,#2}(#3)}
\def\wloc#1#2#3{W_{\scriptstyle loc}^{#1,#2}(#3)}
\def\osc{\mathop{\rm osc}}
\def\var{\mathop{\rm Var}}
\def\supp{\mathop{\rm supp}}
\def\Cap{{\rm Cap}}
\def\norma#1#2{\|#1\|_{#2}}
\def\diam{\mathop{\rm diam}}

\def\C{\Gamma}

\let\text=\mbox

\catcode`\@=11
\let\ced=\c
\def\a{\alpha}
\def\b{\beta}
\def\c{\gamma}
\def\d{\delta}
\def\g{\lambda}
\def\o{\omega}
\def\q{\quad}
\def\n{\nabla}
\def\s{\sigma}
\def\div{\mathop{\rm div}}
\def\sing{{\rm Sing}\,}
\def\singg{{\rm Sing}_\ty\,}

\def\A{{\cal A}}
\def\F{{\cal F}}
\def\H{{\cal H}}
\def\W{{\bf W}}
\def\M{{\cal M}}
\def\N{{\cal N}}

\def\eR{{\bf R}}
\def\eN{{\bf N}}
\def\Ze{{\bf Z}}
\def\Qe{{\bf Q}}
\def\Ce{{\bf C}}

\def\ty{\infty}
\def\e{\varepsilon}
\def\f{\varphi}
\def\:{{\penalty10000\hbox{\kern1mm\rm:\kern1mm}\penalty10000}}
\def\ov#1{\overline{#1}}
\def\D{\Delta}
\def\O{\Omega}
\def\pa{\partial}

\def\st{\subset}
\def\stq{\subseteq}
\def\pd#1#2{\frac{\pa#1}{\pa#2}}
\def\sgn{{\rm sgn}\,}
\def\sp#1#2{\langle#1,#2\rangle}

\newcount\br@j
\br@j=0
\def\q{\quad}
\def\gg #1#2{\hat G_{#1}#2(x)}
\def\inty{\int_0^{\ty}}
\def\od#1#2{\frac{d#1}{d#2}}

\def\bg{\begin}
\def\eq{equation}
\def\bgeq{\bg{\eq}}
\def\endeq{\end{\eq}}
\def\bgeqnn{\bg{eqnarray*}}
\def\endeqnn{\end{eqnarray*}}
\def\bgeqn{\bg{eqnarray}}
\def\endeqn{\end{eqnarray}}

\def\bgeqq#1#2{\bgeqn\label{#1} #2\left\{\begin{array}{ll}}
\def\endeqq{\end{array}\right.\endeqn}

\def\abstract{\bgroup\leftskip=2\parindent\rightskip=2\parindent
        \noindent{\bf Abstract.\enspace}}
\def\endabstract{\par\egroup}

\def\udesno#1{\unskip\nobreak\hfil\penalty50\hskip1em\hbox{}
             \nobreak\hfil{#1\unskip\ignorespaces}
                 \parfillskip=\z@ \finalhyphendemerits=\z@\par
                 \parfillskip=0pt plus 1fil}
\catcode`\@=11

\def\cal{\mathcal}
\def\eR{\mathbb{R}}
\def\eN{\mathbb{N}}
\def\Ze{\mathbb{Z}}
\def\Qu{\mathbb{Q}}
\def\Ce{\mathbb{C}}

\def\osd{\mathrm{osd}\,}

\def\sdim{\mbox{\rm s-dim}\,}
\def\sd{\mbox{\rm sd}\,}

\section{Introduction} The history of the study of singularities in the context of various 
function spaces and boundary value problems is extremely rich. 
Among numerous authors who investigated singular sets of Sobolev functions and of weak solutions of elliptic equations, we
mention Deny, Lions, Fuglede, Aronszajn, Smith, Serrin, Reshetnyak, Stein, Havin, Mazya, Bagby, Ziemer, Meyers, Veron, Mou, Grillot and
Kilpel\"ainen. The mentioned work can be found among the references in \cite{cras}.
In this paper, we are interested in the problem of concentration of singularities 
near a given point for a weak solution of the
 simplest Laplace boundary value problem (BVP)
\bgeq\label{laplace}
-\D u=F(x),\,\,u\in H_0^1(\O),
\endeq
where $\O$ is a given bounded open subset of $\eR^N$, and $F\in L^2(\O)$. 
We consider functions $F$ in (\ref{laplace})
which generate {\em maximally singular} solutions $u:\O\to\ov\eR$ (to be defined below), and study the question of their 
pointwise regularity and irregularity (see Theorem \ref{cor}). 
In order to dimensionally measure a singular set, we use the Hausdorff dimension. Here, we recall the property of {\em countable stability}
of the Hausdorff dimension which will be frequently used in our proofs through this article. 
Let $A_1, A_2,\ldots$ be a countable sequence of subsets of $\eR^N$, then one has
\begin{align}
 \dim_H \Big( \bigcup_{i=1}^\infty A_i \Big) = \sup_{1 \leq i < \infty} \big( \dim_H A_i \big).
 \label{CountStab}
\end{align}
Fractal dimensions used in this paper (namely, the Hausdorff and box dimensions) are defined and studied, e.g., in \cite{falc}.

Let us first define the notion of maximally singular solution, 
introduced in
\cite{chaos}; see also \cite{lana}.
Let $u:\O\to\ov\eR$ be a given Lebesgue measurable function, where $\O$ is an open subset of $\eR^N$. By $\sing u$ we denote the {\em singular set} of $u$, i.e.\ the set of
points $a\in\O$ such that there exist positive constants $C=C(a)$, $\c=\c(a)$ and $r=r(a)$ satisfying the condition $u(x)\ge C|x-a|^{-\c}$
for a.e.\ $x\in B_r(a)$. Here, $B_r(a)$ denotes the open ball in $\eR^N$ of radius $r$. The singular set is well defined, 
since if $u=v$ a.e.\ in $\O$, then clearly $\sing u=\sing v$.

Let $X=X(\O)$ be any given space (or just a nonempty set) of Lebesgue measurable functions $u:\O\to\ov\eR$.
We define {\em singular dimension} of $X$ by
\bgeq\label{sing}
\sdim X=\sup\{\dim_H(\sing u):u\in X\}.
\endeq
This definition has been introduced in \cite{cras}. Here, $\sdim X\in[0,N]$.
If the supremum is achieved for some $v\in X$, then $v$ is said to be {\em maximally singular} function in~$X$.


\begin{figure}
	\centering
	\includegraphics[width=60mm, height=60mm]{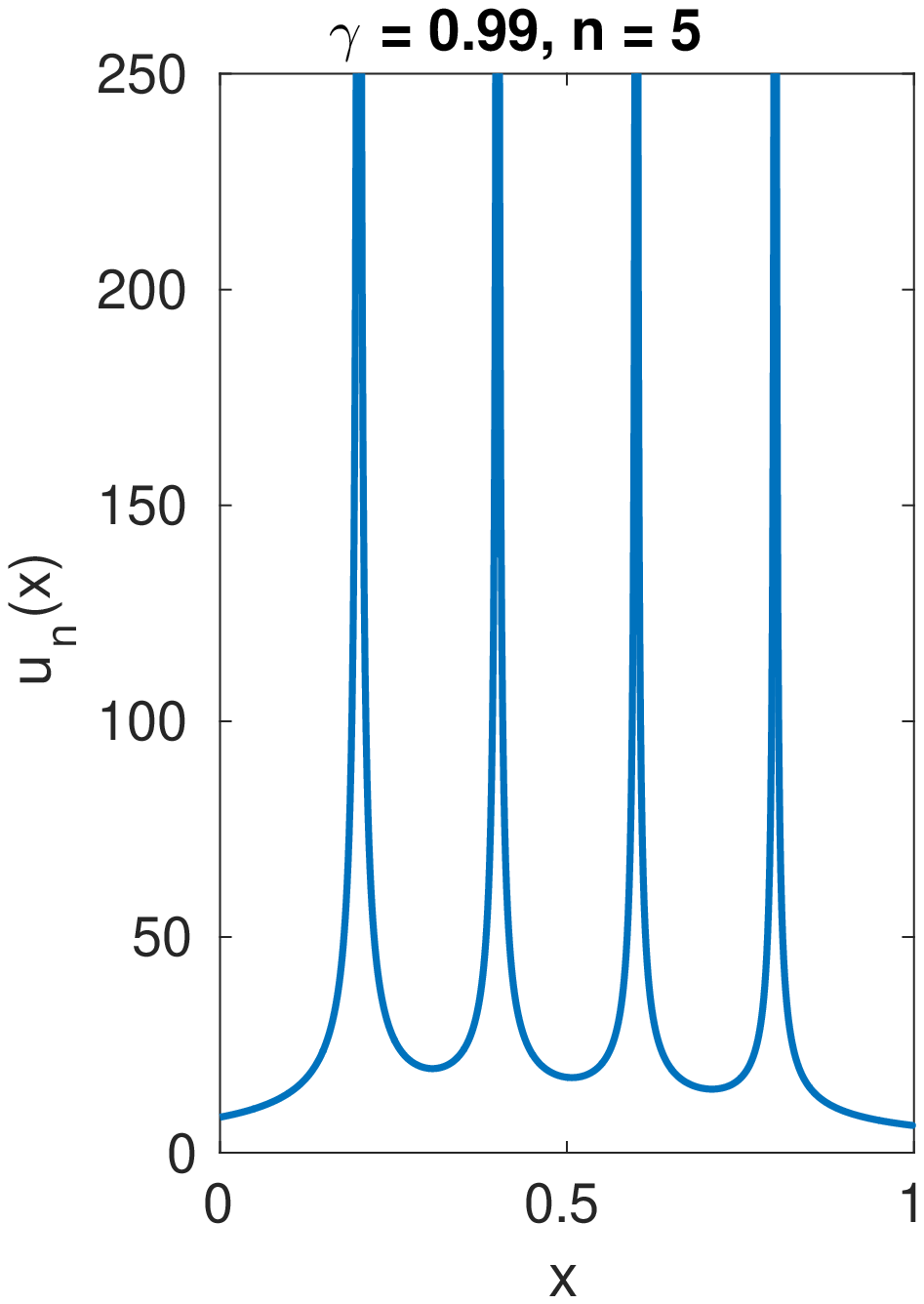}
	\includegraphics[width=60mm, height=60mm]{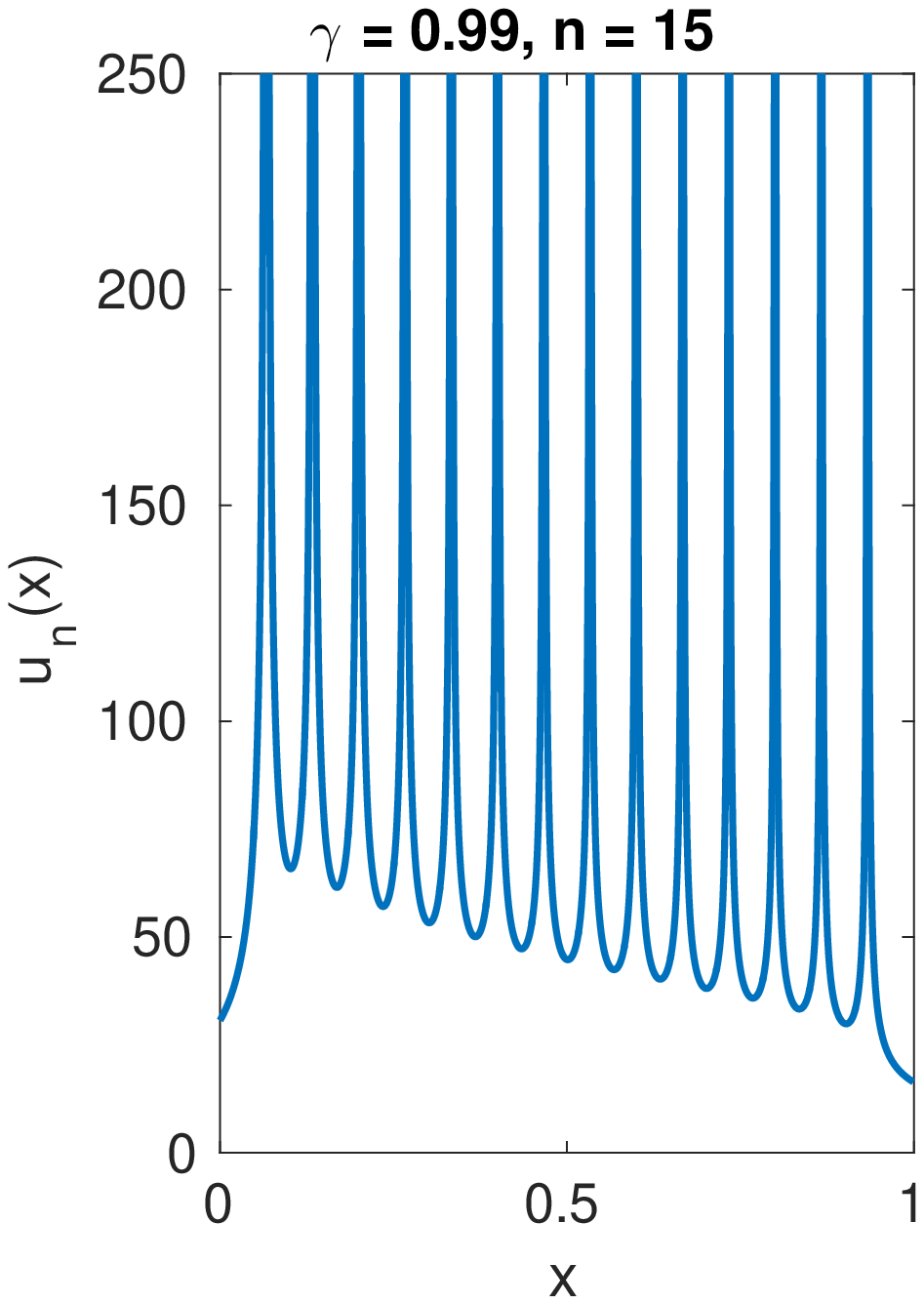}
	\caption{There are Lebesgue integrable functions $u:\O\to\ov\eR$, 
	with $\O$ being a bounded open subset of $\eR^N$, such that $\dim_B(\sing u)=N$, 
	that can be explicitly constructed as $ u(x) = \sum_{k=1}^\ty c_k |x - a_k|^{-\gamma}$, 
	where $(a_k)_k$ is a dense subset of $\O$ and $\c\in(0,N)$. 
	Since the graph of the function $u$ cannot be drawn (namely, its singular set $\sing u$ is dense $\O$) even for $N=1$ and $\O=(0,1)$, 
	we draw the function $u_n:(0,1)\to\ov\eR$ instead, defined by $ u_n(x) = \sum_{k=1}^n c_k |x - a_k|^{-\gamma}$,
	where $\di c_k = k^{-1/2}$, $a_k = \frac{k}{n}$ where $k=1,2,\ldots,n-1$.
	See Remark \ref{fig1}.}
	\label{Fig1}
\end{figure}

\begin{remark}\label{fig1} 
Figure \ref{Fig1} and the caption below indicate that it has no sense to define the singular 
dimension $\sdim X(\O)$ introduced in \eqref{sing} via the (upper) box dimension instead of the 
Hausdorff dimension. This has already been noticed in \cite{chaos}, which we reproduce here from the sake of completeness.
Note that for the function $u:\O\to\eR$ described in the caption to Figure~\ref{Fig1}, 
we have that $\dim_B(\sing u)=\dim_B(\ov{\sing u})=\dim_B(\ov\O)=N$, so the value of $\sdim X(\O)$ would be either zero or $N$. 
The construction of the function $u$ is based on the Stein's trick; see Stein \cite[Ch.\ 5, \S6.3]{stein}.
\end{remark}

We recall that $\sdim L^p(\O)=N$ for $1\le p<\ty$, $\sdim W^{k,p}(\O)=(N-kp)^+$ if $1<p<\ty$ (here $t^+:=\max\{0,t\}$ is the positive part of $t$),
see \cite{cras}.
Moreover, these spaces contain maximally singular functions, and furthermore, maximally singular functions are dense in respective spaces; see \cite{lana}. 
The analogous claim holds also for Besov spaces, Lizorkin-Triebel spaces, and Hardy spaces; see \cite{archiv}. 
For Besov spaces, see also the results obtained by Jaffard and Meyer in \cite{jm}. 

The main results of this article are stated in Theorem~\ref{Fsums} and Theorem~\ref{0,N-4} below. 
Using a suitable choice of the right-hand sides $F\in L^2(\O)$ of BVP \eqref{laplace}, 
it is possible to generate its weak solutions $u\in H_0^1(\O)$ with the property of 
{\em contrast}; that is, solutions that are $C^2$-regular in arbitrary
prescribed open subset  $\O_r$ of $\Omega$,
while they are {\em maximally singular} in {\em any} open subset of $\O\setminus\ov{\O_r}$; see Theorem~\ref{0,N-4}. 

For a given nonempty set $A\st\eR^N$ and $\d>0$, it will be convenient to define $A_\d$ as the open $\d$-neighborhood of $A$;
that is, $A_\d=\{x\in\eR^N:d(x,A)<\d\}$, where $d(x,A)$ is the Euclidean distance from $x$ to $A$.


\section{Some auxiliary results}

In the following lemma, we introduce a suitable class of measurable functions $F:\O\to\ov\eR$ that are 
singular in a subset $A$ of $\O$ (in the sense that $\sing F\supseteq A$) and are locally Lipschitz continuous in $\O\setminus \ov A$. 
These functions will enable us to generate the corresponding weak solutions of BVP \eqref{laplace} that are singular 
in any point of $A$ and $C^2$-regular in $\O\setminus\ov A$; see Theorem \ref{Fsums} below for a more precise formulation.

\begin{lemma}\label{lip}
Let $\O\stq\eR^N$ be a bounded open set. Let $F:\O\to\ov\eR$ be a function of the form 
$$
F(x)=\sum_{k=1}^\ty c_kd(x,A_k)^{-\c_k},
$$
where  $A_k$ are compact subsets of $\O$ such that $A:=\cup_kA_k$ is not dense in $\O$, $(\c_k)_{k\in\eN}$ 
is a bounded sequence of positive real numbers, and $\sum_k|c_k|<\ty$.
 Then, the function $F$ is locally Lipschitz in $\O\setminus\ov A$. If $c_k>0$ for all $k\in\eN$, then 
$A\stq\sing F\stq \ov A$.
\end{lemma}

\begin{proof}
{\em Part} 1.
First, let us show that the series corresponding to $F(x)$ is absolutely convergent for each $x\in\O\setminus \ov A$. 
Indeed, take any $x\in\O\setminus \ov A$.  Since $A_k\stq A$ then $d(x,A_k)\ge d(x,A)=d(x,\ov A)>0$. 
Denoting $\c=\sup_k\c_k$,  from $0<\c_k\le\c<\ty$
and assuming $d(x,A)\le 1$, we obtain that
$$
\sum_{k=1}^\ty |c_k|d(x,A_k)^{-\c_k}\le d(x,A)^{-\c}\sum_{k=1}^\ty |c_k|<\ty.
$$
If $d(x,A)\ge1$, then, similarly as above, we have that $\sum_{k=1}^\ty |c_k|d(x,A_k)^{-\c_k}\le\sum_{k=1}^\ty |c_k|$.

In what follows we use the fact that for any $\a \neq 0$ and  $u,v>0$, we have that 
\begin{equation}\label{lagrange}
|u^{\a}-v^{\a}|\le|\a|\,\max\{u^{\a-1},v^{\a-1}\}|u-v|,
\end{equation}
which is an easy consequence of the Lagrange mean value theorem. It suffices to show that $F$
is Lipschitz on any set of the form $\O\setminus \ov{A_\d}$ for $\d>0$ small enough, where $\ov{A_\d}=\{x\in\eR^N:d(x,A)\le \d\}$ 
is the closure of $A_\d$. For convenience, we take
$0<\d<\min\{1,\diam(\O)\}$, where $\diam(\O):=\max\{|x-y|:x,y\in\O\}$ denotes the diameter of $\O$.
Indeed, observe that $\O\setminus\ov A$ is equal to the union of the open sets $\O\setminus \ov{A_\d}$ for all indicated values of $\d$; that is
\begin{equation}\label{union}
\O\setminus\ov A=\bigcup_{\d\in(0,\min\{1,\diam(\O)\})}(\O\setminus \ov{A_\d}).
\end{equation}
In other words, it suffices to verify that the function $F$ is
Lipschitz continuous on $\O\setminus\ov{ A_\d}$ for any $\d\in(0,\min\{1,\diam(\O)\})$.

Let $x,y\in\O\setminus \ov{A_\d}$ be fixed. From $d(x,A_k)\ge d(x,A)>\d$, and similarly for $y$, using Eq.\ \eqref{lagrange} we have that
\bgeqn
|d(x,A_k)^{-\c_k}-d(y,A_k)^{-\c_k}|&\le&\c\,\max\{d(x,A_k)^{-\c_k-1},d(y,A_k)^{-\c_k-1}\}\,|d(x,A_k)-d(y,A_k)|\nonumber\\
&=&\c\,\min\{d(x,A_k),d(y,A_k)\}^{-\c_k-1}\,|d(x,A_k)-d(y,A_k)|\nonumber\\
&\le&\c\,\min\{1,\d\}^{-\c-1}\,|x-y|\nonumber
\endeqn
for any $k\in\eN$,
where we have also used a well known fact that the distance function $x\mapsto d(x,A_k)$ has 
the Lipschitz property with the Lipschitz constant equal to $1$; that is, $|d(x,A_k)-d(y,A_k)|\le|x-y|$.
Hence, we conclude that
\bgeqn
|F(x)-F(y)|&\le&
\sum_{k=1}^{\ty}|c_k|\,|d(x,A_k)^{-\c_k}-d(y,A_k)^{-\c_k}|\nonumber\\
&\le&\c\,\min\{1,\d\}^{-\c-1}\left(\sum_{k=1}^{\ty}|c_k|\right)|x-y|.\nonumber
\endeqn
In other words, $F$ is Lipschitz continuous on the open set $\O\setminus\ov{A_\d}$, and the Lipschitz constant of the function $F$, 
restricted to this set, is not exceeding the value of 
\[ \c\,\min\{1,\d\}^{-\c-1}\left(\sum_{k=1}^{\ty}|c_k|\right)<\ty. \] 
Hence, due to \eqref{union}, the function $F$ is locally Lipschitz on the set $\O\setminus\ov A$.
\smallskip

{\em Part} 2.
Concerning the last claim in Lemma \ref{lip}, in which we assume that all of the constants $c_k$ are positive, 
for $k\in\eN$, the inclusion $A\stq\sing F$ is obvious. The second inclusion $\sing F\stq\ov A$ follows at once from {\em Step}~1.

This completes the proof of the lemma.
\end{proof}

As another auxiliary result, we shall need the following interesting Lebesgue integrability condition for functions $F:\O\to\ov\eR$ of the form 
$F(x):=d(x,A)^{-\c}$, which involves the upper Minkowski 
(or box) dimension of $A$, denoted by $\ov\dim_BA$. Since in the sequel, the set $A$ will always be of the Lebesgue measure zero, 
then $F$ is well defined a.e.\ in $\O$. 

\begin{prop}[Harvey--Polking, \cite{hp}]\label{hp} Let $\O$ be a nonempty bounded and open subset of $\eR^N$. Assume that
$A$ is a nonempty bounded subset of $\O$ such that $\ov A\st\O$ and $\c$ is a nonzero real number. 
\begin{itemize}
\item[($a$)] If $\c < N-\ov\dim_BA$, then $d(\,\cdot\,,A)^{-\c}\in L^1(\O)$.
\item[($b$)] If $p > 1$ and $\c<\frac1p(N-\ov\dim_BA)$, then $d(\,\cdot\,,A)^{-\c}\in L^p(\O)$.
\end{itemize}
\end{prop}

Case ($a$) of Proposition \ref{hp} is contained in \cite[p.\ 42]{hp}, but stated equivalently in terms of the Minkowski content of $A$.
For a detailed discussion about the Harvey--Polking result, and its role in the study of complex dimensions of fractal sets, 
see \cite{laprazu1}.  
The claim in case ($b$) of the proposition follows immediately from ($a$). In this paper, 
the exponent $\c$ will always be positive, so that in both ($a$) and $(b)$ we 
necessarily have $\ov\dim_BA<N$, and hence, $A$ is of the Lebesgue measure zero.
\medskip

 In order to generate weak solutions $u$ of BVP \eqref{laplace} on a prescribed
nonempty subset $A$ of $\O$, we consider functions $F\in L^2(\O)$ such that $F(x)\ge d(x,A)^{-\c}$.

\begin{theorem}\label{hpr} Let $\O$ be a bounded open subset of $\eR^N$ and $N\ge5$.
Assume that a function $F\in L^2(\O)$ is such that $F(x)\ge C d(x,A)^{-\c}$ a.e.\ in $\O$, where $A$ is a compact subset of $\O$, $C>0$ and $\c>0$ such that 
\begin{equation}
 2 < \c < \frac{1}{2} \Big( N - \ov\dim_BA \Big).
 \label{HP_LD}
\end{equation}
Then for the corresponding weak solution $u\in H_0^1(\O)$ of \eqref{laplace}, we have that
\begin{equation}
A\stq \sing u.
\end{equation}
Moreover, for any point $a\in A$, the solution $u$ of \eqref{laplace} has singularity of order $\c-2>0$ at $a$. More specifically, there are two positive constants $C_1$ and $C_2$, such that
\begin{equation}
u(x)\ge C_1d(x,A)^{-(\c-2)}-C_2 \; \textrm{ a.e.  in }\; \O.
\end{equation}
\end{theorem}

\begin{proof}[Sketch of the proof] (For details we refer to \cite[for $p=2$]{gen2000} and \cite{lana}.) Using Prop.~\ref{hp}, we see that the right-hand side inequality in \eqref{HP_LD} implies that $d(\,\cdot\,,A)^{-\c}\in L^2(\O)$. 

On the other hand, the left-hand side inequality of \eqref{HP_LD} gives that $A \subseteq \sing u$, where $u$ is the weak solution of \eqref{laplace}.
In order to show this, let us fix any $a\in A$, and let $r>0$ be such that $B_r(a)\stq\O$. Let $u_1$ be a weak solution of the following auxiliary BVP:
\bgeq\label{laplace1}
-\D u_1=C|x-a|^{-\c},\,\,u_1\in H_0^1(B_r(a)),
\endeq
By direct computation (solving the corresponding spherically symmetric BVP problem $-\D v=C|x|^{-\c}$ in $B_r(0)$ and $v=0$ on $\pa B_r(0)$, via the associated ODE), we see that $u_1(x)=v(x+a)$ can be explicitly obtained in the following form: 
$$
u_1(x)= C_1|x-a|^{-(\c-2)}-C_2.
$$ 
Since $F(x)\ge F_1(x):=C|x-a|^{-\c}$ a.e.\ in $B_r(a)$ and $u\ge u_1$ on $\pa B_r(a)$ (note that the right-hand side $F$ of \eqref{laplace} is nonnegative, so that $u\ge0$ a.e.\ in $\O$; also, note that the set $A$ is of the Lebesgue measure zero, since necessarily $\ov\dim_BA<N$), using the comparison principle we conclude that $u(x)\ge u_1(x)=C_1|x-a|^{-(\c-2)}-C_2$ a.e.\
in $B_r(a)$, which proves that $a\in\sing u$.
\end{proof}

\begin{remark} In Theorem \ref{hpr} we can recognize a regularizing effect of the Laplace BVP. Namely, if the right-hand side $F\in L^2(\O)$ of the BVP \eqref{laplace} has a singularity of order $\c>2$ at a point $a\in\O$, then the corresponding weak solution $u\in H_0^1(\O)$ of \eqref{laplace} has a singularity of smaller order, equal to $\c-2$. It is interesting that for the corresponding $p$-Laplace BVP $-\D_p u=F(x)$ with $u\in W_0^{1,p}(\O)$ and $F\in L^{p'}(\O)$, where $p':=p/(p-1)$, the regularizing property does not hold when $p\ne 2$. We have the phenomenon of the so called {\em loss of regularity of weak solutions}; see \cite{dea}.\end{remark}

The following result deals with the regularity of maximally singular solutions of the standard Laplace BVP \eqref{laplace}.
Here, condition $N\ge5$ cannot be dropped; that is, Theorem \ref{Fsums} is not true for $N=1,2,3$ and~$4$. 
Indeed, for any $F\in L^2(\O)$, the weak solution $u$ of BVP \eqref{laplace} is contained in the Sobolev space $H^2(\O)$, 
so that by the Sobolev embedding theorem (see, e.g., \cite{brezis} or \cite{gt}), for $N\le3$ the solutions are continuous, 
while for $N=4$ solutions of \eqref{laplace} cannot have (strong) singularities, however, singularities of weaker type (say logarithmic) are possible in the letter case.
In particular, for $N=1,2,3$ and~$4$ we have $\sing u=\emptyset$.

\begin{theorem}\label{Fsums}
Let $\O$ be a bounded open set in $\eR^N$ of class $C^2$, where $N\ge5$, and let $F:\O\to\ov \eR$ be defined by
\bgeq\label{Fsuma}
F(x)=\sum_{k=1}^\ty \frac{c_k}{\|d(\,\cdot\,,A_k)^{-\c_k}\|_{L^2}}d(x,A_k)^{-\c_k},
\endeq
where $A_k$ are subsets of $\O$ such that  $\ov{A_k}\st\O$ for any $k\in\eN$, and $c_k$ are positive constants such that
\bgeq\label{sums}
\sum_{k=1}^\ty c_k<\ty\q\mbox{\rm and}\q\sum_{k=1}^\ty \frac{c_k}{\|d(\,\cdot\,,A_k)^{-\c_k}\|_{L^2}}<\ty,
\endeq
and
\bgeq\label{dim}
2<\c_k<\frac1{2}(N-\ov\dim_BA_k),\q\lim_{k\to\ty}\dim_HA_k=N-4.
\endeq
Then, the weak solution $u$ of the elliptic BVP $-\D u=F(x)$, $u\in H_0^1(\O)$, is maximally singular, 
i.e., $\dim_H(\sing u)=N-4$, and $u\in H^2(\O)$. Furthermore,
$u$ is singular on $A:=\cup_{k=1}^\ty A_k$ and regular on $\O\setminus\ov A$ in the following sense: 
\begin{equation*}
A\stq\sing u\stq \ov A\q\mbox{\rm and}\q u\in C_{loc}^{2,\a}(\O\setminus\ov A),
\end{equation*} 
for any $\a\in(0,1)$. 
More specifically, $u\in C^{2,\a}(\O\setminus A_\d)$ for any $\d>0$ and any $\a\in(0,1)$.
\end{theorem}

\begin{proof}  The first condition in (\ref{sums}) and the second inequality in (\ref{dim}) imply that $F\in L^2(\O)$. Hence, 
the corresponding weak solution $u$ of \eqref{laplace} is maximally singular; see \cite[Theorem 2]{lana}.
According to Agmon, Douglis and Nirenberg regularity result (see, e.g., Brezis \cite[Th\'eor\`eme IX.32]{brezis}),
the solution is in the Sobolev space $H^2(\O)$.

On the other hand, the second condition in (\ref{sums}) and Lemma~\ref{lip} imply that $F$ is 
locally Lipschitz on $\O\setminus\ov A$. Hence, $F$ is of class $C^{\a}_{loc}$ on this set
for any $\a\in(0,1)$. Since $u\in H^2(\O)$, by Gilbarg and Trudinger \cite[Theorem~9.19, p.~243]{gt}, 
it follows that the solution $u$ is in fact
of class $C^{2,\a}_{loc}$ on $\O\setminus\ov A$. 
\end{proof}

\smallskip

\begin{remark}\label{ck}
Both inequalities in (\ref{sums}) can be fulfilled if we assume that $c_k\to0^+$ sufficiently rapidly
as $k\to\ty$. For example, it suffices to take the sequence $(c_k)_{k\in\eN}$ such that
$$
0<c_k\le\min\{2^{-k},2^{-k}\|d(\,\cdot\,,A_k)^{-\c_k}\|_{L^2}\}\q\mbox{for all $k\in\eN$}.
$$
Furthermore, we recall that we have $\dim_HA_k\le\ov\dim_HA_k$ for all $k\ge1$; see, e.g., \cite{falc}. 
Hence, since $\ov\dim_BA_k\le N-4$ by \eqref{dim}, we then also have $\lim_{k\to\ty}\ov\dim_BA_k=N-4$.
\end{remark}
\medskip

In Theorem \ref{cor} below, we consider the case when the singular set $\sing u$ is 
dense in any prescribed open subset $\O_s$ of $\O$, while $u$ is of class $C^2$ in $\O\setminus \ov{\O_s}$.

\begin{theorem}\label{cor} 
Let $\O$ be a given nonempty bounded open subset of $\eR^N$, with $N\ge 5$ and
let $\O_s$ be any prescribed nonempty open subset of $\O$. Then, there exists an explicit function 
$F\in L^2(\O)$ such that the singular set $\sing u$ of the associated weak solution $u\in H_0^1(\O)$ of \eqref{laplace} 
  is dense in $\O_s$ $($that is, $\ov{\sing u} = \ov{\O_s}$ $)$, while on the other hand,
 $u\in C_{loc}^{2,\a}(\O\setminus \ov{\O_s})$ for any $\a\in(0,1)$. 
 In addition to this, we can achieve that the solution $u$ of \eqref{laplace} is maximally singular, i.e., $\dim_H(\sing u)=N-4$.
\end{theorem}

\begin{proof}
Let $(a_k)_{k\in\eN}$ be a sequence of points which is dense in $\O_s$, and let $(r_k)_{k\in\eN}$ be a sequence of positive real numbers, 
such that $B_k=B_{r_k}(a_k)\stq \O_s$ for all $k\in\eN$. Let us choose the subsets $A_k$ of $\eR^N$ 
such that $\dim_BA_k$ exists for any $k\ge1$, $\dim_BA_k=\dim_HA_k < N-4$ and $\di \lim_{k\to\ty}\dim_HA_k  = N-4$. 

For every $k\in\eN$, the set $A_k$ can be constructed as the generalized {\em Cantor grill}; that is, 
$A_k=R_k(C_k\times[0,1]^{N-5})$, where $C_k\st[0,1]$ is the generalized Cantor set (see \cite{falc}) 
with the common value of the Minkowski and Hausdorff dimensions being smaller than and arbitrarily close to $1$, 
and $R_k$ is the composition of a suitable scaling and rigid motion, such that $A_k\st B_k$.

Next, for each $k\in\eN$, we can choose the corresponding real number $\c_k$ satisfying the first condition in \eqref{dim}. 
This choice of $\gamma_k$ is possible since $\dim_B A_k<N-4$ and $N\ge5$; that is, the open interval $\big(2,\frac12(N-\overline{\dim}_B A_k)\big)$, 
in which we choose $\c_k$, is nonempty.

Finally, we take a sequence of positive numbers $(c_k)_{k\in\eN}$ satisfying condition \eqref{sums}, for example, as in Remark \ref{ck}.
The claim then follows from Theorem \ref{Fsums}, by defining
$F(x)$ explicitly by Eq.\ \eqref{Fsuma}.
\end{proof}

\section{Pointwise maximally singular functions}\label{msf}

The result given in Theorem \ref{0,N-4} concerns nonuniform regularity  of a weak solution $u:\O\to \overline{\eR}$ of the Laplace equation
\eqref{laplace} in the following sense:
there exist two disjoint open subsets $\O_r$ and $\O_s$ of $\O$ such that $\ov\O=\ov{\O_r}\cup \ov{\O_s}$, and
the restricted function $u|_{\O_r}$ has no singularities, while $u|_{\O_s}$ is singular in a dense subset of $\O_s$.
Before stating Theorem \ref{0,N-4}, we start with definitions of the main concepts: 
 the pointwise maximally singular functions and the singular dimension of a function at the given point.

\begin{defn}  {\sf (Pointwise maximally singular functions)}\newline%
Assume that $X(\O)$ is a given space $($or just a nonempty set$)$ of Lebesgue measurable real functions 
$u:\O:\to\ov\eR$ defined on a nonempty open subset $\O$ of $\eR^N$.
We say that a function $u\in X(\O)$ is {\em pointwise maximally singular} $($with respect to $X(\O)$$)$
in a given open subset $\O_s\st\O$ if for any open ball $B\st\O_s$ the function $u$ 
is maximally singular on $B$; that is,
\begin{equation*}
\dim_H(\sing (u|_B)) = \sdim X(\O).
\end{equation*}
\end{defn}

In particular, if 
\begin{equation}\label{XO}
X(\O):=\{ u \in  H_0^1(\O):-\D u=F(x)\,\,\mbox{in $\mathcal{D}'(\O)$, for $F\in L^2(\O)$}\}
\end{equation} 
is the solution set corresponding to BVP \eqref{laplace} and $N\ge5$, then a given function $u\in X(\O)$ is pointwise maximally
singular on a prescribed open subset $\O_s$ of $\O$ if $\dim_H(\sing u|_B)=N-4$ for any open ball $B\st\O_s$. Here, we recall that
\bgeq\label{N-4}
\sdim X(\O)=(N-4)^+,
\endeq
and if $N\ge5$, then there exist maximally singular functions in $X(\O)$ (maximally singular functions have been introduced just after Eq.\ \eqref{sing}); see \cite{lana}. 
Note that Theorem \ref{cor} represents a refinement of this result.

One can also imagine the functions $u\in X(\O)$ with varying regularity from point to point.
In this sense, we introduce the following definition.

\begin{defn} {\sf (Singular dimension of a function at a given point)}\newline%
Let $u:\O\to\ov\eR$ be a Lebesgue measurable function.
 For any $a\in\ov\O$ we define {\em singular dimension of  $u$ at the point $a$} by 
\bgeq\label{sd}
(\sd u)(a):=\lim_{r\to0}\dim_H(\sing (u|_{B_r(a)\cap\O}))
\endeq
\end{defn}

The value of $(\sd u)(a)$ represents a numerical measure of concentration of singular points of $u$
near the point $a\in\ov\O$, in the sense of Hausdorff dimension. 
In this way, we have obtained the singular dimension function, associated with $u$, such that
\begin{equation*}
\sd u:\ov\O\to[0,N]. 
\end{equation*}
Let $X$ be any given space (or just a nonempty set) of Lebesgue measurable functions.
It is clear that $0\le(\sd u)(a)\le\sdim X \le N$
for all $u\in X$ and $a\in\ov\O$. 
We say that a function $u\in X$ is {\em maximally singular at the point} $a$ $($with respect to $X$ $)$ if 
$$
(\sd u)(a)=\sdim X.
$$

\begin{remark} The function $\sd u:\ov\O\to[0,N]$ may be nonconstant, as Theorem~\ref{0,N-4} below shows. 
We can address the following open problem. We ask if it is possible to construct a function 
$F\in L^2(\O)$ generating a weak solution $u\in X(\O)= H_0^1(\O)\cap H^2(\O)$, with
$N\ge5$, of the Laplace equation (\ref{laplace}), 
such that $\sd u:\ov\O\to[0, N-4]$ is equal to any prescribed, continuous and surjective function from $\ov\O$ to $[0,N-4]$?
More generally, how do the functions $\sd u$ look like for $u\in X(\O)$?
In Lemma~\ref{upper} below, we show that the function $\sd u:\ov\O\to[0,N]$ is upper semicontinuous 
for any measurable function $u:\O\to \overline{\eR}$. See also the open problems stated near the end of Section \ref{op} below.
\end{remark}

The following theorem states that, under certain assumptions
on the dimension of $\eR^N$, for any two prescribed disjoint open subsets of $\O$ there exists $F\in L^2(\O)$ such that
the corresponding weak solution $u \in H_0^1(\O)$ of $\di -\D u = F(x)$ is of class $C_{loc}^{2,\alpha}$ on one part of the domain
while it is pointwise maximally singular in the closure of the other part. 

\begin{theorem}\label{0,N-4} {\sf (Maximal contrast of weak solutions of BVP \eqref{laplace})}\newline%
Let $\O$ be a bounded open subset of $\eR^N$ of class $C^2$, where $N\ge5$.
Let $\O_r$ and $\O_s$ be any two prescribed disjoint open subsets of $\O$ such that 
$\ov\O=\ov{\O_r}\cup\ov{\O_s}$. 
Then there exists $F\in L^2(\O)$ such that
for the corresponding weak solution $u\in H_0^1(\O)$ of problem \eqref{laplace} one has:
\begin{align}
 (\sd u)(a) = 
\begin{cases}
 0 & \textrm{ for all } a \in \O_r,\\
 N-4 & \textrm{ for all } a \in \ov{\O_s}.
\end{cases}
\label{Thm3.Rez1}
\end{align}
\end{theorem}

\begin{proof}  
The construction of functions $F \in L^2(\Omega)$ relies on the form given in \eqref{Fsuma}.

Let $(B_j)_{j \in \eN}$ be a countable base of open balls in $\O_s$. Let $(d_k)_{k\in\eN}$
be an increasing sequence of positive real numbers such that $d_k\in(N-5,N-4)$ and $d_k\to N-4$ as $k\to\ty$. For each $k$ there exists 
a family of sets $\{A_{kj}\st\O:j\in \eN\}$, such that
\medskip 

($a$) $A_{kj}\st B_j$ for all $j$,
\smallskip 

($b$) $d_k=\dim_H A_{kj}=\dim_BA_{kj}$ for all $j$.
\medskip 

For any fixed positive integer $k$, the sequence of sets $(A_{kj})_{j\in \eN}$, with $j\in\eN$, satisfying properties ($a$) and ($b$), can be constructed
using a fixed set $A_k$ of both Hausdorff and box dimensions equal to $d_k$. 
To this end let us first  define
an auxiliary set 
$$
A_k=C_k\times [0,1]^{N-5}\st\eR^{N-4}\st\eR^N,
$$ 
where $C_k\st[0,1]$ is a generalized Cantor set such that $\dim_BC_k=\dim_HC_k=d_k-(N-5)\in(0,1)$; see Falconer \cite{falc} for their construction. 
Here, if $N=5$, the set $[0,1]^{N-5}$ is identified with a point. 
Then, by scaling and translating  $A_k$, we can easily construct a sequence of sets
$A_{kj}$, with $j\in \eN$, satisfying the desired properties $(a)$ and $(b)$. Recall that the operations of scaling and translating of the set 
$A_k$ keep both the Hausdorff and box dimensions unchanged; see Falconer \cite{falc}.
Therefore, using the well known additivity property of the Hausdorff dimension with respect to 
Cartesian products of sets (see Falconer \cite{falc}), we conclude that
$$
\dim_HA_{kj}=\dim_HA_k=\dim_HC_k+(N-5)=d_k,
$$
and analogously $\dim_BA_{kj}=d_k$, for all $k,j\in\eN$. 

Now, we 
choose a sequence $(\c_k)_{k\in \eN}$ of real numbers such that
$$
2<\c_k<\frac1{2}(N-d_k)\q\mbox{for all $k$}.
$$
Note that, for each $k\ge1$, the open interval $\big(2,\frac1{2}(N-d_k)\big)$ in which we choose $\c_k$ is nonempty, since $d_k<N-4$.
We have that for each $j\in\eN$,
$$
\dim_HA_{kj}=d_k\to N-4\q\mbox{as $k\to\ty$.} 
$$
We can also achieve that the sequence $(d_k)_{k\in \eN}$ is nondecreasing.
The function $F\in L^2(\O)$ is then constructed analogously as in the proof of
Theorem \ref{Fsums}: 
\bgeq\label{Fs}
F(x)=\sum_{k,j=1}^\ty \frac{c_{kj}}{\|d(\,\cdot\,,A_{kj})^{-\c_k}\|_{L^2}}\,d(x,A_{kj})^{-\c_k}
\endeq
where the double sequence $(c_{kj})_{k,j \in \eN}$ is chosen so that  $c_{kj}>0$ and
\bgeq\label{sumskk}
\sum_{k,j=1}^\ty c_{kj}<\ty,\q\sum_{k,j=1}^\ty \frac{c_{kj}}{\|d(\,\cdot\,,A_{kj})^{-\c_k}\|_{L^2}}<\ty.
\endeq
If $u\in H_0^1(\O)$ is a weak solution of the corresponding distribution equation $-\D u=F(x)$, then due to Theorem~\ref{Fsums} we have that for any fixed $j\in\eN$,
\bgeq\label{uj}
\bigcup_{k=1}^\infty A_{kj} \stq \sing (u|_{B_j}).
\endeq
The countable stability of the Hausdorff dimension (along with the fact that the sequence $(d_k)_{k\in \eN}$ is nondecreasing) implies that
$$
\dim_H\Big(\bigcup_{k=1}^\infty A_{jk}\Big)=\sup_k(\dim_HA_{kj})=\sup_k d_k=
\lim_{k\to\ty}d_k=N-4.
$$
Therefore, due to Eq.\ \eqref{uj} and since (by the regularity theory for weak solutions of elliptic BVPs; see, e.g.,
\cite{brezis}) $u\in H^2(\O):=W^{2,2}(\O)$, we conclude that for any fixed $j\in\eN$, 
\begin{equation}
\begin{aligned}
N-4&=\dim_H\Big(\bigcup_{k=1}^\infty A_{kj}\Big)\le\dim_H\big(\sing (u|_{B_j})\big)\\
&\le\dim_H(\sing u)\le\sdim H^2(\O)=N-4,
\end{aligned}
\end{equation}
where in the last equality we have used the fact that $\sdim W^{k,p}(\O)=N-kp$, provided $p>1$, $k\in\eN$ and $kp<N$; 
see \cite{cras} (here, we take $k=p=2$).
Hence, 
$$
\dim_H(\sing u|_{B_j})=N-4,
$$ 
for any $j\in\eN$. Since $(B_j)_{j \in \eN}$ is the base of neighborhoods of $\O_s$, the solution $u$
corresponding to $F$ is maximally singular in any point $a\in\ov{\O_s}$; that is,
$(\sd u)|_{\ov{\O_s}}\equiv N-4$, while $(\sd u)|_{\O_r}\equiv0$, since $\sing u|_{\O_r} = \emptyset$.
In this way we proved Eq.\ \eqref{Thm3.Rez1}. 
This completes the proof of the theorem. 
\end{proof}

\begin{cor} \label{(iii)}
 With the assumptions given by Theorem~\ref{0,N-4}, there exists a function $F\in L^2(\O)$ such that the corresponding weak 
 solution of the BVP $-\D u=F(x)$, $u\in H_0^1(\O)$, 
 is pointwise maximally singular in the whole of $\ov\O$, i.e. $(\sd u)(a) = N-4$ for all $a\in\ov\O$.
\end{cor}

\begin{proof}
 Let us define $\O_r:=\emptyset$ and $\O_s:=\O$. 
 The claim follows immediately from the proof of Theorem~\ref{0,N-4}. Here, we note that Lemma~\ref{lip} holds for $\O_s = \O$ and
$\O_r = \emptyset$ as well. 
\end{proof}

\begin{remark}
 We point out that in Theorem \ref{0,N-4}, we have that $\sd u\equiv \sd F$. Their common value is either $0$ or $N-4$, 
depending on whether $a \in \ov \O \setminus \ov{\O_s}$ or $a \in \ov \O \setminus\O_r$. In general, for weak solutions of \eqref{laplace}, 
we have that $\sd u \le \sd F$ (since $\sing u\stq\sing F$), and the inequality may be strict. 
It is easy to construct a nonnegative $L^2$-function $F$ as in \eqref{Fsuma}, such that
$\sd F\equiv N$, while $\sd u\equiv 0$. It suffices to take $\c_k\in(0,2)$ in~\eqref{Fsuma}, 
with $\cup_k A_k$ being dense in $\O$. 
\end{remark}

As a consequence of Theorem~\ref{0,N-4}, we provide the following corollary which states that
the Lebesgue spaces, Sobolev spaces, and Bessel potential spaces possess pointwise maximally singular functions.

\begin{cor} Assume that $\O$ is an open set in $\eR^N$, not necessarily bounded.
Let $X=L^p(\O)$ $($with $1\le p<\ty$ $)$, or $W^{k,p}(\O)$ $($with
$1<p<\ty$, $kp<N$ $)$, or $L^{\a,p}(\eR^N)$ $($with $1<p<\ty$,  $0<\a p<N$ $)$.  
Then, $X$ possesses pointwise maximally singular functions $u\in X$; i.e.,
there exist functions $u\in X$ such that $(\sd u)(a)=\sdim X$ for all $a\in\ov\O$.
Moreover, these functions form a dense subset in $X$.
$($For $X=W^{k,p}(\O)$ 
we also assume that the domain $\O$ satisfies the Lipschitz property or, more generally, that it is a 
{\rm Sobolev extension domain}; i.e., $X=W^{k,p}(\O)$ can be continuously embedded into the space $W^{k,p}(\eR^N)$ $)$.
\end{cor}

\begin{proof} From the construction of maximally singular function in $X$, we know for any prescribed open ball $B$
 there exists nonnegative maximally singular function $u\in X$ and a set $A$ such that $A \stq \sing u$, 
 $\dim_HA = \sdim X$, and $A\st B$ (see the proof of Theorem~\ref{0,N-4}).  
 In order to construct a pointwise maximally singular function in $X$, let $(B_j)_{j\ge 1}$ be a base of open balls in $\O$,
 and let $u_j\in X$ be nonnegative function such that $A_j\stq\sing u_j$, $\dim_HA_j=\sdim X$,
 $A_j\st B_j$. In particular, each $u_j$ is maximally singular. 
 Then it is easy to verify that the function $\di u = \sum_j2^{-j}\frac{u_j}{\|u_j\|_X}$ is pointwise maximally singular.

Let $X = L^p(\O)$ and let $u_0$ be a pointwise maximally singular function in $X$.
Since $C_0^{\ty}(\O)$ is dense in $X$, then 
\begin{equation}
\bigcup_{k=1}^\ty\big(k^{-1}u_0+C_0^{\ty}(\O)\big)
\end{equation}
 is a dense subset of $X$, consisting of pointwise maximally singular functions. Indeed,
 each of the functions $k^{-1}u_0+v$ is pointwise maximally singular for any function~$v\in C_0^{\ty}(\O)$ and any $k\ge1$.
The case of $X=L^{\a,p}(\eR^N)$ is treated similarly.
 
 If $X = W^{k,p}(\O)$, then we proceed analogously using the fact that the space $C_0^{\ty}(\eR^N)|_{\O}$ is dense
 in $W^{k,p}(\O)$ if $\O$ is a Lipschitz domain (or more generally, a {\em Sobolev extension domain}); see Adams \cite[p.\ 54]{adams}.
\end{proof}

\section{Upper semicontinuity of the singular dimension function and open problems}\label{op}

 We conclude this article with one observation about the singular dimension function $\sd u:\ov\O\to[0,N]$, introduced by Eq.~\eqref{sd}.

\begin{lemma}\label{upper}
Let $u\:\O\to\eR$ be a Lebesgue measurable function. Then, the singular dimension function $\sd u:\ov\O\to[0,N]$
is upper semicontinuous; that is, for any $a\in\ov\O$,
$$
(\sd u)(a)\ge\limsup_{a_k\to a}(\sd u)(a_k).
$$
\end{lemma}

\begin{proof}
Assume the contrary. Then, there exists $\e>0$ and a subsequence $(k')$ of $(k)$ such that
$$
(\sd u)(a_{k'})\ge\e+(\sd u)(a).
$$
Hence, there exists the sequence of radii $r(k')$ small enough, so that
\bgeq\label{k'}
\dim_H(\sing (u|_{B_{r(k')}(a_{k'})\cap\O}))\ge\frac\e2+(\sd u)(a).
\endeq
For any fixed ball $B_r(a)$, there exists $k'$ large enough such that
$B_{r(k')}(a_{k'})\st B_r(a)$. Taking $\limsup$ as $k'\to\ty$ in 
$$
\dim_H(\sing (u|_{B_{r(k')}(a_{k'})})\le 
\dim_H(\sing (u|_{B_{r}(a)})),
$$ 
and then passing to the limit as $r\to0$, we obtain that
$$
\limsup_{k'\to\ty}\dim_H(\sing (u|_{B_{r(k')}(a_{k'}))\cap\O})\le (\sd u)(a).
$$
But this contradicts~(\ref{k'}).
\end{proof}

Knowing that a function $\sd u:\ov\O\to[0,N]$ is upper semicontinuous, it naturally arises
the following question, which we state as an open problem.

Let $X=X(\O)$ be a given space (or set) of real functions (for example, let 
$X=L^2(\O)$, or $X=W^{k,p}(\O)$ with $kp<N$, or let $X$ be defined via the BVP \ref{laplace} by Eq.\ \eqref{XO}).
Given any upper semicontinuous function $f:\O\to[0,\sdim X]$, is there a function $u\in X$
such that $\sd u  \equiv f$?
\medskip

 For a conclusion of this article, we point out that one can study analogous questions as in this paper in the context of the standard $p$-Laplace BVP:
\bgeq\label{plaplace}
-\D_p u=F(x),\,\,u\in W_0^{1,p}(\O),
\endeq
with $1<p<\ty$ and $F\in L^{p'}(\O)$, where $\O$ is a bounded open subset of $\eR^N$, 
$\D_p u:=\div(|\n u|^{p-2}\n u)$ and $p':=p/(p-1)$ is the conjugate exponent associated with $p$. 
We do not know if the solution set  
\begin{equation}\label{XOp}
X(\O,p):=\{ u \in  W_0^{1,p}(\O):-\D_p u=F(x)\,\,\mbox{in $\mathcal{D}'(\O)$, for $F\in L^{p'}(\O)$}\},
\end{equation}
corresponding to BVP \eqref{plaplace}, possesses maximally singular functions for $p\ne2$. 
However, it can be shown that if $p\ge2$, then (see \cite{gensing}):
\bgeq\label{N-pp'}
\sdim X(\O,p)=(N-pp')^+.
\endeq
This result extends the result stated in \eqref{N-4}. Furthermore, it is not yet known if \ \eqref{N-pp'} holds when $p\in(1,2)$ as well.


\end{document}